\documentclass[12pt]{amsart}
\usepackage{amsmath,amsthm,amssymb}
\usepackage[margin=1.1in]{geometry}
\usepackage{hyperref}
\usepackage[capitalise]{cleveref}
\usepackage{graphicx}
\usepackage{lipsum}
\usepackage[utf8]{inputenc}
\usepackage{tikz}
\newtheorem{thm}{Theorem}
\newtheorem{theorem}[thm]{Theorem}

\newtheorem{proposition}{Proposition}
\newtheorem{example}{Example}

\newtheorem{definition}{Definition}
\newtheorem{problem}{Problem}
\newtheorem{conjecture}{Conjecture}
\theoremstyle{remark}
\newtheorem{remark}{Remark}

\def\Cone{{\rm Cone}}
\def\Spec{{\rm Spec}}
\def\a{{\mathbf{a}}}
\def\e{{\mathbf{e}}}
\def\x{{\mathbf{x}}}
\def\X{{\mathbf{X}}}
\def\y{{\mathbf{y}}}
\def\z{{\mathbf{z}}}
\def\Y{{\mathbf{Y}}}
\def\W{{\mathbf{W}}}
\def\v{{\mathbf{v}}}
\def\V{{\mathbf{V}}}
\def\Int{{\rm Int}}
\def\R{{\mathbb R}}
\def\bM{{\mathcal{M}}}

\def\PGL{{\rm PGL}}
\def\Aut{{\rm Aut}}

\def\dlog{{\rm dlog}}
\def\Vol{{\rm Vol}}

\def\C{{\mathbb C}}
\def\Z{{\mathbb Z}}
\def\Gr{{\mathrm{Gr}}}

\def\pt{{\mathrm{pt}}}
\def\H{{\mathcal{H}}}
\def\RR{{\mathcal{R}}}
\def\PP{{\mathbf{P}}}

\renewcommand\P{{\mathbb P}}
\def\Res{{\rm Res}}

\def\oPi{\mathring{\Pi}}
\def\w{{\mathbf{w}}}
\def\wo{{\mathbf{w_0}}}
\def\t{{\mathbf{t}}}

\author{Thomas Lam}
\title{An invitation to positive geometries}
\begin{document}

\begin{abstract}
This short introduction to positive geometries, targeted at a mathematical audience, is based on my talk at OPAC 2022.
\end{abstract}

\maketitle

Positive geometries are certain semialgebraic spaces, equipped with a distinguished meromorphic form called the canonical form \cite{ABL}.  Examples of positive geometries include polytopes, positive parts of toric varieties, totally nonnegative spaces, and Grassmann polytopes and amplituhedra.

Positive geometries were first introduced in theoretical physics.  Canonical forms of positive geometries are used to write formulae for scattering amplitudes, analytic functions that are used to compute probabilities in particle scattering experiments.  Roughly speaking, different positive geometries correspond to different quantum field theories.  There is a flourishing and well-developed industry expanding the zoo of positive geometries and related physical processes.  We refer the reader to the recent surveys \cite{FL,HT} for more on the physics of positive geometries, and to \cite{EH,HP} for textbook introductions to scattering amplitudes.  %See also \cite{SAGEX}

The mathematical study of positive geometries is still in its infancy.  In this short note, we give a brief introduction to positive geometries, with a mathematical audience, especially an algebraic or geometric combinatorialist, in mind.  The course webpage \cite{Lamcourse} contains many additional references for further exploration.

We apologize for the multiple perspectives, especially those from physics, that we do not mention.  We thank the OPAC organizers, Christine Berkesch, Ben Brubaker, Gregg Musiker, Pavlo Pylyavskyy, and Vic Reiner for the invitation to speak.  We thank the National Science Foundation for support under grant DMS-1953852.

\section{The polytope canonical form}\label{sec:poly}
We begin by illustrating the main ideas with the example of convex projective polytopes.  A \emph{projective polytope} is a subspace $P \subset \P^d(\R)$ of projective space such that there exists a hyperplane $H \subset \P^d$ and a linear identification $\P^d \setminus H \cong \R^d$ such that $P \subset \P^d \setminus H \cong \R^d$ is a Euclidean convex polytope.  An orientation of a projective polytope is an orientation of its interior $\Int(P)$, which is always a manifold.  The following theorem is a reformulation of the statement that polytopes are positive geometries \cite[Section 6]{ABL}.

\begin{thm}[Residue definition]\label{thm:Pdef}
For each full-dimensional oriented projective polytope $P \subset \P^d$ there exists a rational $d$-form $\Omega(P)$ on $\P^d$, with poles only along facet hyperplanes, and these poles are simple, and such that $\Omega(P)$ is uniquely determined by the recursive properties:
\begin{enumerate}
\item the canonical form of a point $\pt = \P^0$ is $\Omega(\pt)=\pm 1$ depending on the orientation,
\item for any facet $F \subset P$, we have $\Res_H \Omega(P) = \Omega(F)$, where $H$ is the hyperplane spanned by the facet $F$.
\end{enumerate}
\end{thm}
The differential form $\Omega(P)$ is called the \emph{canonical form} of $P$.  The notation $\Res_H \Omega(P)$ denotes a residue, defined in \eqref{eq:resdef}; we illustrate the computation in some examples below.  

In the following examples, if $P \subset \R^d$ is a Euclidean polytope, then we consider it a projective polytope via the natural inclusion $\R^d =\{(1:\x) \mid \x \in \R^d\} \subset \P^d$.  Let $P = [a,b] \subset \R \subset \P^1(\R)$ be a closed interval with one of the two orientations.  Then the canonical form is (up to a sign) given by
\begin{equation}\label{eq:interval}
\Omega([a,b]) = \frac{dx}{x-a}-\frac{dx}{x-b},
\end{equation}
which has residue $+1$ at $x=a$ and $-1$ at $x=b$.

\begin{example}\label{ex:main}
Let $P$ be the quadrilateral with vertices $(0,0), (2,0), (0,1),(1,2)$ in $\R^2$ (\cref{fig:ex}(a)).  Then 
\begin{equation}\label{eq:P}
\Omega(P) = \frac{4+4x-y}{xy(1+x-y)(4-2x-y)} dx dy.
\end{equation}
As required, $\Omega(P)$ has simple poles along each of the four facet hyperplanes
\begin{equation}\label{eq:facets}
x=0, \qquad y=0, \qquad 1+x-y=0, \qquad 4-2x-y=0
\end{equation}
of $P$.  The numerator will be explained in \cref{thm:Padjoint} below.  Let us first take the residue along $x=0$.  Since
$$
\Omega(P) = \frac{dx}{x} \wedge \frac{4+4x-y}{y(1+x-y)(4-2x-y)} dy,
$$
we have, using \eqref{eq:interval},
$$
\Res_{x=0}(\Omega(P))= \frac{4-y}{y(1-y)(4-y)}dy = \frac{1}{y(1-y)}dy= \frac{1}{y} dy - \frac{1}{y-1}dy = \Omega([0,1]).
$$
Now, let us take the residue along the hyperplane $H=\{f=1+x-y=0\}$.  Since $df/f = (dx-dy)/(1+x-y)$, we can write
$$
\Omega(P) = \frac{df}{f} \wedge \frac{4+4x-y}{xy(4-2x-y)} dx,
$$
Making the substitution $y=x+1$, we get
$$
\Res_{H}(\Omega(P))= \left( \frac{4+4x-y}{xy(4-2x-y)} dx \right)_{y \mapsto 1+x}=\frac{3+3x}{x(x+1)(3-3x)} dx = \frac{1}{x(1-x)}dx,
$$
which again by \eqref{eq:interval} is the canonical form of the edge connecting the vertices $(0,1)$ and $(1,2)$.  We invite the reader to verify that the residues along the two other facets are correct.
\end{example}

\begin{figure}
  \includegraphics[width=1.0\textwidth]{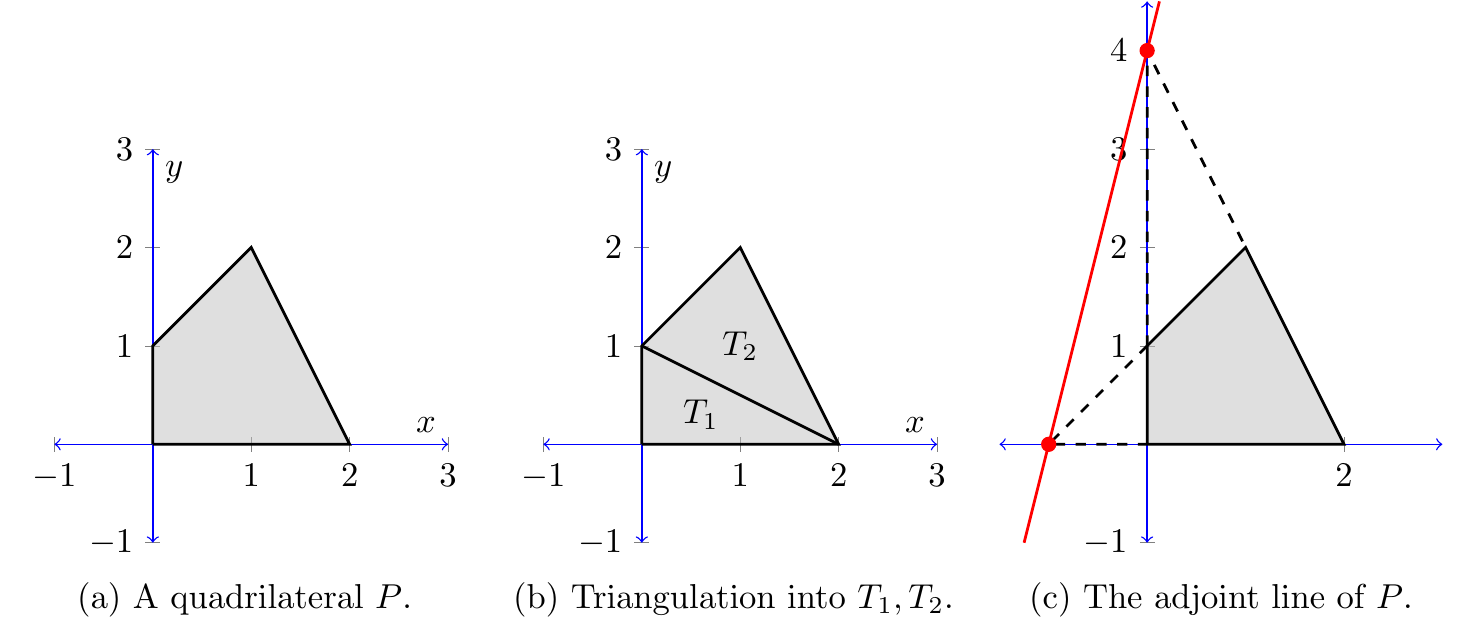}
 \caption{\label{fig:ex} Our running example.}
\end{figure}

Note that changing the orientation of $P$ negates the canonical form $\Omega(P)$.  Henceforth, we omit the adjective ``oriented" from our theorem statements, implicitly assuming that our polytopes have been equipped with (compatible) orientations.

The uniqueness of $\Omega(P)$ follows from the fact that $\P^d$ has no nonzero holomorphic $d$-forms.  The quickest proof of the existence of $\Omega(P)$ is via triangulations (\cite[Section 3]{ABL}).

\begin{thm}[Additivity for subdivisions]\label{thm:Ptriang}
Suppose that a projective polytope $P$ is subdivided into polytopes $T_1,T_2,\ldots,T_r$.  Then
$$
\Omega(P) = \sum_{i=1}^r \Omega(T_r).
$$
\end{thm}
Since every polytope has a triangulation, to construct canonical forms, it suffices to construct the canonical form of a simplex.  Let 
$$
\Delta^d := \{(X_0:X_1:\cdots:X_d) \in \P^d \mid X_i \geq 0\} \subset \P^d
$$
denote the standard projective $d$-simplex.  Every face of $\Delta^d$ is again a standard simplex, so it is easy to verify that 
$$
\Omega(\Delta^d) = \frac{dx_1}{x_1} \wedge \cdots \wedge \frac{dx_d}{x_d}, \qquad x_i:=\frac{X_i}{X_0}
$$
satisfies the recursion in \cref{thm:Pdef}.

\begin{example} Let $P$ be as in Example~\ref{ex:main}.  Then $P$ can be triangulated as in \cref{fig:ex}(b).  The two canonical forms are 
  \begin{equation}
    \label{eq:triang}
    \Omega(T_1) = \frac{2}{xy(2-x-2y)}dxdy, \qquad \Omega(T_2) = \frac{9}{(1+x-y)(4-2x-y)(2-x-2y)}dxdy
  \end{equation}
which can be calculated by finding a projective transformation $g \in \PGL(3) \cong \Aut(\P^2)$ sending $T_1$ (or $T_2$) to the standard simplex $\Delta^2$ and computing the pullback $g^*(\Omega(\Delta^2))$.  More explicitly, suppose the triangle $T$ has facets $a_i+b_ix+c_iy=0$ for $i =1,2,3$.  Then letting $g$ be the $3 \times 3$ matrix with columns $(a_i,b_i,c_i)$, we have 
$$
\Omega(T)=\pm \frac{\det(g)}{\prod_i (a_i+b_ix+c_iy)}.
$$
It is easily verified that indeed $\Omega(T_1)+\Omega(T_2)$ recovers $\Omega(P)$ in \eqref{eq:P}, in agreement with \cref{thm:Ptriang}.
\end{example}

We give an explicit formula for $\Omega(P)$.  For a subset $S \subset \R^d$, the polar set $S^\vee$ is defined by
$$
S^\vee:= \{\x \in \R^d \mid \x.\y \geq -1 \text{ for all } \y \in S\}.
$$

\begin{thm}[Dual volume {\cite[Section 7.4]{ABL}}]\label{thm:Pdual} Suppose $P \subset \R^d \subset \P^d$ is full-dimensional.  Then 
\begin{equation}\label{eq:Pdual}
\Omega(P)(\x) = \Vol((P-\x)^\vee) d^d\x.
\end{equation}
where $\x \in \Int(P)$.  Here, $\Vol$ denotes the Euclidean volume, normalized so that the unit simplex has volume $1$.
\end{thm}
The function $\Vol((P-\x)^\vee)$ is defined when $\x \in \Int(P)$. It analytically continues to a rational function on $\R^d$.
While \eqref{eq:Pdual} depends on (compatible) choices of inner product and measure, we emphasize that the definition (Theorem~\ref{thm:Pdef}) does not depend on any metric notions.

\begin{example}
The polar polytope $(P-x)^\vee$ is illustrated in \cref{fig:polar}.  The vertices of $(P-x)^\vee$ lie on the rays of the (inner) normal fan of $P$.  Intersecting $(P-x)^\vee$ with each maximal cone of the normal fan we obtain a triangle.  Summing the volumes of these triangles gives
$$
\frac{1}{xy} + \frac{1}{1+x-y}+ \frac{3}{(4-2x-y)(x+1-y)} + \frac{2}{y(4-2x-y)} = \frac{4+4x-y}{xy(1+x-y)(4-2x-y)},
$$
agreeing with \cref{thm:Pdual}.
\end{example}

\begin{figure}
  \includegraphics[width=0.7\textwidth]{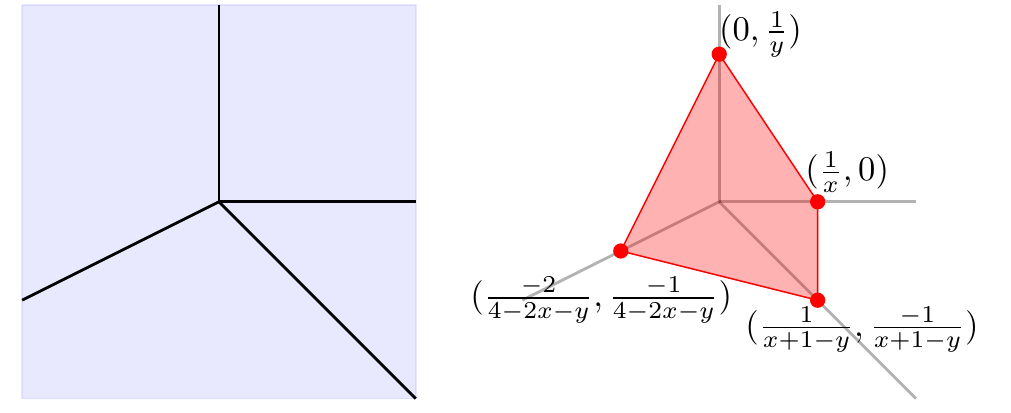}
 \caption{\label{fig:polar} Left: the normal fan of $P$. \;\; Right: the polar polytope $(P-(x,y))^\vee$.}
\end{figure}

\begin{remark}
The rational function $\Vol((P-\x)^\vee)$ is a special case of the \emph{dual mixed volume} function.  Let $P_1,P_2,\ldots,P_r \subset \R^d$ be Euclidean polytopes.  The \emph{mixed volume} is the polynomial
$$
V_\PP(\x):= \Vol(x_1 P_1+ \cdots + x_r P_r)
$$ 
where $P+Q$ denotes the Minkowski sum of polytopes $P$ and $Q$.  We define the \emph{dual mixed volume} function by
$$
V_\PP^\vee(\x):= \Vol((x_1 P_1+ \cdots + x_r P_r)^\vee).
$$ 
As before, this is defined when $0 \in \Int(x_1 P_1+ \cdots + x_r P_r)$ and extended by analytic continuation.  The function $V^\vee(\x)$ is a rational function.   Choosing $P_i = - \e_i$ and $P_{d+1} = P$, and setting $x_{d+1} = 1$, one sees that $\Vol((P-\x)^\vee)$ is a special case of $V^\vee(\x)$.  See also \cite{AHL} for an appearance of the dual mixed volume function.
\end{remark}

Define the $d$-form $ \langle \X d^d \X \rangle$ on $\R^{d+1}$,
$$
 \langle \X d^d \X \rangle := d! \sum_{i=0}^d (-1)^i X_i dX_0 \wedge \cdots \wedge \widehat{dX_i} \wedge \cdots \wedge dX_d.
$$
Under the tautological rational map $\tau:\R^{d+1} \dashrightarrow \P^d$, the pullback $\tau^*(\Omega)$ of a rational $d$-form $\Omega$ on $\P^d$ can be written as $\tau^*(\Omega) = \frac{P(\X)}{Q(\X)}  \langle \X d^d \X \rangle$, where $\frac{P(\X)}{Q(\X)}$ is a rational function, homogeneous of degree $-d-1$.  %We abuse notation by identifying a rational form $\Omega$ on $\P^d$ with its pullback to $\R^{d+1}$.  

For a cone $C \subset \R^{d+1}$, define the dual cone $C^\vee$ by 
$$
C^\vee = \{\x \in \R^{n+1} \mid \x.\y \geq 0 \text{ for all } \y \in C\}.
$$

\begin{thm}[Laplace transform {\cite[Section 7.4]{ABL}}]\label{thm:PLaplace} 
Let $P \subset \P^d$ be a projective polytope, and let $C(P) \subset \R^{d+1}$ be the (pointed, convex, polyhedral) cone over $P$.  Then
$$
\tau^*\Omega(P) = \frac{1}{d!} \left(\int_{C^\vee} e^{-\X.\Y} d^{d+1}\Y\right) \langle \X d^d \X \rangle
$$
where $C^\vee \subset \R^{d+1}$ is the dual cone, and $\X,\Y$ are vectors in $\R^{d+1}$.  Here, the integral converges when $\X \in C(P)$, and is extended by analytic continuation to $\R^{d+1}$.
\end{thm}

\begin{example}\label{ex:dual}
  Continuing \cref{ex:main}, the cone $C(P)$ over $P$ is given by
\begin{equation}\label{eq:cone}
C(P) = {\rm Cone}\left(\begin{bmatrix}1\\0\\0\end{bmatrix},\begin{bmatrix}1\\2\\0\end{bmatrix},\begin{bmatrix}1\\1\\2\end{bmatrix},\begin{bmatrix}1\\0\\1\end{bmatrix}\right)
\end{equation}
with dual cone
$$
C^\vee = {\rm Cone}\left(\a_1=\begin{bmatrix}0\\1\\0\end{bmatrix},\a_2=\begin{bmatrix}1\\1\\-1\end{bmatrix},\a_3=\begin{bmatrix}4\\-2\\-1\end{bmatrix},\a_4=\begin{bmatrix}0\\0\\1\end{bmatrix}\right)
$$
whose generators correspond to the four facets \eqref{eq:facets}.  The cone $C^\vee$ can be triangulated $C^\vee = C_1 \cup C_2$ into two simplicial cones $C_1 = \Cone(\a_1,\a_2,\a_4)$ and $C_2= \Cone(\a_2,\a_3,\a_4)$, and we compute that

\begin{align}
  \begin{split}\label{eq:Laplace}
\int_{C_1} e^{-\X.\Y} d^{d+1} \Y &= \frac{|\det(\a_1,\a_2,\a_4)|}{X_1(X_0+X_1-X_2)X_2} = \frac{1}{X_1(X_0+X_1-X_2)X_2} \\
\int_{C_2} e^{-\X.\Y} d^{d+1} \Y &= \frac{|\det(\a_2,\a_3,\a_4)|}{(X_0+X_1-X_2)(4X_0-2X_1-X_2)X_2} = \frac{6}{(X_0+X_1-X_2)(4X_0-2X_1-X_2)X_2}.
  \end{split}
\end{align}

Summing then specializing to $(X_0,X_1,X_2)=(1,x,y)$, we obtain \eqref{eq:P}, agreeing with \cref{thm:PLaplace}.
\end{example}
We remark that the calculation of \cref{ex:dual} comes from a triangulation of the dual cone (or dual polytope) and is qualitatively quite different from the one in \cref{thm:Pdual} (e.g., unlike \eqref{eq:triang}, both rational functions in \eqref{eq:Laplace} have poles belonging to the facets of $P$).  Indeed, the volume of a polytope $P$ can be obtained by triangulating $P$ or by triangulating the dual of $P$, a phenomenon sometimes known as Filliman duality~\cite{Filliman}.

\medskip

Let $P$ be a full-dimensional polytope in $\P^d$ with $f$ facets.  We define the \emph{adjoint hypersurface} $A_P \subset \P^d$ as follows.  Define $\H_P$ to be the projective hyperplane arrangement consisting of all facet hyperplanes of $P$.  We say that $\H_P$ is \emph{simple} if through any point in $\P^d$ pass at most $d$ hyperplanes of $\H_P$.  The \emph{residual arrangement} $\RR_P$ consists of all linear subspaces that are intersections of hyperplanes in $\H_P$ that do not contain faces of $P$.

When $\H_P$ is simple, Kohn and Ranestad \cite{KR} proved that there is a unique hypersurface $A_P$ in $\P^d$ of degree $f-d-1$ which vanishes along the residual arrangement $\RR_P$.  In general (for $P$ not necessarily simple), we write $P = \lim_{t \to 0} P_t$ for polytopes $P_t$ such that $\H_{P_t}$ is simple for $t >0$, and define the adjoint hypersurface $A_P:= \lim_{t \to 0} A_{P_t}$.  In the following, we abuse notation by also using $A_P$ to denote a polynomial whose vanishing set is the adjoint hypersurface.  See \cite{Warren} for a different definition of adjoint.

The following result connects the canonical form with adjoint hypersurfaces; see the learning seminar notes of Gaetz \cite{Gaetz}.
\begin{thm}[Zeros equal adjoint hypersurface]\label{thm:Padjoint}
Suppose $P \subset \R^d \subset \P^d$.  Then 
$$
\Omega(P) = \alpha \cdot \frac{A_P}{\prod_{\text{facets } H} H} d^d\x
$$
for some nonzero constant $\alpha$.
\end{thm}
\cref{thm:Padjoint} can be proved by triangulating $P$ and using \cref{thm:Pdual}. 
%see  and the discussion in {\cite[Section 7.1]{ABL}}.

We remark that any rational form $\P^d$ has degree $-d-1$ since $\omega_{\P^d} \simeq {\mathcal O}_{\P^d}(-d-1)$.  Since all the poles of $\Omega(P)$ are linear, this agrees with the degree of $A_P$ being equal to $f-d-1$.

If $P$ is a simplex, then the adjoint hypersurface is empty, $A_P$ is a constant, and $\Omega(P)$ has no zeroes.

\begin{example} The rational 2-form \eqref{eq:P} has four poles, which are the four facets of $P$.  The numerator vanishes along the line connecting $(-1,0)$ to $(0,4)$, the adjoint hypersurface of $P$.  See \cref{fig:ex}(c).
\end{example}

Let $\pi: X \dashrightarrow Y$ be a rational map between complex algebraic varieties of the same dimension.  In the following, we shall use the notion of the \emph{pushforward} $\pi_* \omega$ of a rational differential form $\omega$ on $X$; see \cite[(4.1)]{ABL}.

\begin{thm}[Pushforward {\cite[Section 7.3]{ABL}}] \label{thm:Ppush} Suppose $P$ has vertices $\W_1,\W_2,\ldots,\W_m \in \P^d$ viewed as vectors in $\R^{d+1}$ in the following.  Let $\V_1=(1,\v_1),\V_2=(1,\v_2),\ldots,\V_m=(1,\v_m) \in \Z^{d+1}$ be integer vectors with first coordinate equal to $1$.  Assume that $\W_i$ and $\V_i$ have the same oriented matroid, that is, ${\rm sign} \det(\W_{i_1},\ldots,\W_{i_{d+1}}) = {\rm sign} \det(\V_{i_1},\ldots,\V_{d+1}) \in \{-1,0,1\}$ for all $\{i_1,\ldots,i_{d+1}\}$.  Define the rational map $\Phi: \P^d \to \P^d$ by
$$
\Phi: (1:\z) = (1:z_1: \cdots:z_d) \longmapsto \sum_{i=1}^m \z^{\v_i} \W_i \in \P^d.
$$
  Then 
$$
\Omega(P) = \Phi_*\Omega(\Delta^d).
$$
\end{thm}
In \cref{thm:Ppush}, the rational map $\Phi$ has degree equal to the normalized volume of the polytope with vertices $\V$, and maps $\Int(\Delta^d)$ homeomorphically to $\Int(P)$.
\cref{thm:Ppush} also has an interpretation in terms of toric varieties; see \cref{sec:toric}.

\begin{example}
Choose $\W_1,\W_2,\W_3,\W_4 \in \R^3$ to be the generators of $C(P)$ in \eqref{eq:cone}.  We will choose $\V$ to be a unit square, with $\V_1 = (1,0,0), \V_2 = (1,1,0), \V_3 = (1,1,1), \V_4 = (1,0,1)$, which clearly has the same oriented matroid as $\W$.  The map $\Phi$, restricted to $\R^2=\{(z_1,z_2)\}$ is given by
$$
\Phi(z_1,z_2) = \begin{bmatrix}1\\0\\0\end{bmatrix} + z_1 \begin{bmatrix}1\\2\\0\end{bmatrix} + z_1z_2 \begin{bmatrix}1\\1\\2\end{bmatrix} + z_2 \begin{bmatrix}1\\0\\1\end{bmatrix} = \begin{bmatrix}1+z_1+z_1z_2+z_2\\2z_1+z_1z_2\\ 2z_1z_2+z_2 \end{bmatrix} \approx \begin{bmatrix}1\\\frac{2z_1+z_1z_2}{1+z_1+z_1z_2+z_2}\\ \frac{2z_1z_2+z_2}{1+z_1+z_1z_2+z_2} \end{bmatrix}.
$$
Using a symbolic computation package, one computes that there are two solutions to $(x,y) = (\frac{2z_1+z_1z_2}{1+z_1+z_1z_2+z_2}, \frac{2z_1z_2+z_2}{1+z_1+z_1z_2+z_2})$, given by 

\noindent
\resizebox{\linewidth}{!}{
  \begin{minipage}{\linewidth}
\begin{align*}
  (z_1^{(1)},z_2^{(1)})=\z^{(1)} &= \left(\frac{\sqrt{2 (x-2) y+(x+2)^2+y^2}-3 x-y+2}{2 (2
  x+y-4)},-\frac{\sqrt{2 (x-2) y+(x+2)^2+y^2}+x-3 y+2}{2
  (x-y+1)}\right)\\
  (z_1^{(2)},z_2^{(2)})= \z^{(2)} &= \left(-\frac{\sqrt{2 (x-2) y+(x+2)^2+y^2}+3 x+y-2}{2 (2
  x+y-4)},\frac{\sqrt{2 (x-2) y+(x+2)^2+y^2}-x+3 y-2}{2 (x-y+1)}\right)
\end{align*} 
\end{minipage}
}

The differential form $\Phi_*(\frac{dz_1dz_2}{z_1z_2})$ is given by summing over the solutions $\z^{(1)}$ and $\z^{(2)}$.  Explicitly, for $i =1,2$, define the Jacobian 
$$
J^{(i)}(x,y) = \det\left(\begin{bmatrix}\frac{\partial z^{(i)}_1}{\partial x} & \frac{\partial z^{(i)}_1}{\partial y}\\
  \frac{\partial z^{(i)}_2}{\partial x} &\frac{\partial z^{(i)}_2}{\partial y} \end{bmatrix}\right).
$$
Then $$
\Phi_*\left(\frac{dz_1dz_2}{z_1z_2}\right) = \left(\frac{J^{(1)}(x,y)}{z_1^{(1)} z_2^{(1)}} + \frac{J^{(2)}(x,y)}{z_1^{(2)} z_2^{(2)}} \right)dx dy
$$
which recovers \eqref{eq:P}.  Note that neither term in the summation is a rational form!
\end{example}
\section{Definition of positive geometry}
Let $X$ be complex $d$-dimensional irreducible algebraic variety, $\omega$ a meromorphic $d$-form on $X$, and $H \subset X$ an (irreducible) hypersurface on $X$.  Assume that $\omega$ has at most simple poles on $H$.  Then the residue $\Res_H \omega$ is the $(d-1)$-form on $H$ defined as follows.  Let $f$ be a local coordinate such that $f$ vanishes to order one on $H$.  Write 
$$
\omega = \frac{df}{f} \wedge \eta + \eta' 
$$
for a $(d-1)$-form $\eta$ and a $d$-form $\eta'$, both without poles along $H$.  Then the restriction
\begin{equation}\label{eq:resdef}
\Res_H \omega:= \eta|_H
\end{equation}
is a well-defined $(d-1)$-form on $H$, not depending on the choices of $f, \eta,\eta'$.

Henceforth, we assume that $X$ is a complex $d$-dimensional irreducible algebraic variety defined over $\R$.  We equip the real points $X(\R)$ with the analytic topology.  Let $X_{\geq 0} \subset X(\R)$ be a closed semialgebraic subset such that the interior $X_{>0} = \Int(X_{\geq 0})$ is an oriented $d$-manifold, and the closure of $X_{>0}$ recovers $X_{\geq 0}$.
% satisfies:
% \begin{enumerate}
%   \item $X_{>0}$ is an oriented $d$-manifold,
%   \item the closure of $X_{>0}$ is $X_{\geq 0}$,
%   \item there is a bijection between connected components of $X_{>0}$ and connected components of $X_{\geq 0}$.  (This)
% \end{enumerate}  
Let $\partial X_{\geq 0}$ denote the boundary $X_{\geq 0} \setminus X_{ \geq 0}$ and let $\partial X$ denote the Zariski closure of $\partial X_{\geq 0}$.  Let $C_1,C_2,\ldots,C_r$ be the irreducible components of $\partial X$.  (We assume that $X$ is generically smooth along each $C_i$, but see \cref{rem:singular}.)  Let 
$C_{i, \geq 0}$ denote the  closures of the interior of $C_i \cap X_{\geq 0}$ in $C_i(\R)$.  The spaces $C_{1,\geq 0}, C_{2,\geq 0},\ldots, C_{r,\geq 0}$ are called the boundary components, or facets of $X_{\geq 0}$.

\begin{definition}\label{def:PG}
  We call $(X,X_{\geq 0})$ a \emph{positive geometry} if there exists a unique nonzero rational $d$-form $\Omega(X,X_{\geq 0})$, called the \emph{canonical form}, satisfying the recursive axioms:
  \begin{enumerate}
    \item If $d = 0$, then $X = X_{\geq 0}= \pt$ is a point and we define $\Omega(X,X_{\geq 0})=\pm 1$ depending on the orientation.
    \item If $d>0$, then we require that $\Omega(X,X_{\geq 0})$ has poles only along the boundary components $C_i$, these poles are simple, and for each $i =1,2,\ldots,r$, we have
    \begin{equation}\label{eq:PGdef}
    \Res_{C_i}\Omega(X,X_{\geq 0})=\Omega(C_i,C_{i,\geq0}).
    \end{equation}
  \end{enumerate}
\end{definition}

\begin{remark}
  While orientations are suppressed in our notation and statements, we always insist that the orientation on a boundary component $C_{i,>0}$ is induced by that of $X_{>0}$.  
\end{remark}

We call $(X,X_{\geq 0})$ \emph{normal} if $X$ is a normal variety and each boundary component $(C_i, C_{i,\geq 0})$ is a normal positive geometry.

\begin{remark}\label{rem:singular}
  \def\tX{{\tilde{X}}}
  \def\tC{{\tilde{C}}}
If $X$ is not normal, one or more irreducible components $C \subset \partial X$ may belong to the singular locus of $X$.  In this case, the definition of boundary component and residue should be modified as follows.  Let $\pi: \tX \to X$ be the blowup of $X$ along the codimension one subvariety $C$.  Define $\tX_{\geq 0} = \overline{\pi^{-1}(X_{> 0} \setminus C)}$ to be the closure of the preimage under $\pi$ of the part of $X_{> 0}$ that does not belong to $C$.  Note that $\pi$ is an isomorphism away from $C$, so $\pi^{-1}(X_{> 0} \setminus C)$ and $X_{> 0} \setminus C$ are diffeomorphic manifolds.  We then define the boundary components of $(X,X_{\geq 0})$ over $C$ as before: let the irreducible components of the Zariski-closure of $(\tX_{\geq 0} \setminus \Int(\tX_{\geq 0})) \cap \pi^{-1}(C)$ be $\tC_1,\ldots,\tC_r$ and let 
$\tC_{i, \geq 0}$ denote the closures of the interior of $\tC_i \cap \tX_{\geq 0}$ in $\tC_i(\R)$.  In \cref{def:PG}, we use $(\tC_i,\tC_{i,\geq 0})$ for $i=1,2,\ldots,r$ in place of $(C,C_{\geq 0})$ and in \eqref{eq:PGdef} we take the residue of $\pi^*\Omega(X,X_{\geq 0})$ along each $\tC_i$.

For example, let $X$ be a rational curve with a node $p \in X(\R)$, and let $\pi:\P^1 \to X$ be the normalization where $a,b \in \P^1(\R)$ both map to $p$.  Let $X_{\geq 0}:= \pi([a,b])$ be the image of the closed interval from $a$ to $b$.   Then $\partial X_{\geq 0} = \{p\}$.  Then $(X,X_{\geq 0})$ is a positive geometry with canonical form the 1-form on $X$ that pullsback to \eqref{eq:interval} on $\P^1$.  This example appears as a boundary of the positive geometry in \cref{fig:polypols}(a).
\end{remark}

For brevity, we may refer to $X_{\geq 0}$ as a positive geometry, and $\Omega(X_{\geq 0})$ its canonical form.  For $(X, X_{\geq 0})$ to be a positive geometry, $X$ cannot have nonzero holomorphic $d$-forms.  Furthermore, if $X$ does not have nonzero holomorphic $d$-forms, then the uniqueness of the canonical form is immediate: the difference of any two such forms would be holomorphic.

Any positive geometry $(C,C_{\geq 0})$ encountered in the recursion \cref{def:PG} is called a \emph{face} of the positive geometry $X_{\geq 0}$.

For $d = 1$, we must have that $X$ is a rational curve.  If $(X,X_{\geq 0})$ is normal, then we have $X \cong \P^1$.  Any finite union of closed intervals in $\P^1(\R)$ is a positive geometry.
For an interval $[a, b]$ (in some affine chart), the canonical form is given by \eqref{eq:interval}.
%$\Omega([a,b])= dx/(x-a)-dx/(x-b)$ 
The canonical form of a disjoint union $\bigcup_i [a_i,b_i]$ of closed intervals is the sum $\sum_i \Omega([a_i,b_i])$.  

Note that $\P^1(\R)$ is not a positive geometry.  We have $\partial \P^1(\R) = \emptyset$, but there is no 1-form on $\P^1$ with no poles.  

Disjoint unions of positive geometries in the same ambient projective algebraic variety are again positive geometries.  
% \begin{example}\label{ex:counter}
% Let $P \subset \R^2$ be the triangle with vertices $(0,0), (1,0), (0,1)$ and $P' \subset \R^2$ be the triangle with vertices $(0,0) (-1,0), (0,-1)$, with the same orientation in the plane.  The union $X_{\geq 0} = P \cup P'$ is not a positive geometry.  One of the boundary components of $\partial X$ is the $x$-axis $C = \{y = 0\}$. The intersection $C \cap X_{\geq 0}$ is the interval $[-1,1]$.  Neither orientation of $[-1,1]$ is compatible with that of $X_{>0}$: we would need to orient $(-1,0)$ one direction, and $(0,1)$ the other direction.  So $(C,C_{\geq 0})$ is not a positive geometry.

% This shows that, in general, a union of positive geometries with disjoint interiors may not be a positive geometry.
% \end{example}

% \begin{remark}
%   The geometry $X_{\geq 0}$ in \cref{ex:counter} has a well-defined canonical form $\Omega(X_{\geq 0}) = \Omega(P)+\Omega(P')$.  For certain applications, we may prefer to expand the definition of positive geometries to include \cref{ex:counter}, allowing boundary components to be formal sums of positive geometries.  In \cref{ex:counter}, $C_{\geq 0}$ would be the formal sum of the two closed interval positive geometries $[-1,0]$ and $[0,1]$ with opposite orientations.
% \end{remark}

\section{Examples of positive geometries}\label{sec:examples}

It follows from \cref{thm:Pdef} that polytopes are positive geometries.  

\subsection{Simplex-like positive geometries}
With \cref{thm:Padjoint} in mind, we call a positive geometry $X_{\geq 0}$ \emph{simplex-like} if $\Omega(X_{\geq 0})$ has no zeroes.  Simplex-like positive geometries are particularly simple since their canonical forms are almost defined without using the recursion of \cref{def:PG}.  Namely, let $\Omega$ and $\Omega'$ be two rational forms, with simple poles along $\partial X$, and no other poles or zeroes.  Then the ratio $\Omega/\Omega'$ is a regular function on the projective variety $X$, and therefore a constant.  Thus the canonical form is defined up to a scalar without any requirement on the residues.

If $(X,X_{\geq 0})$ is a normal simplex-like positive geometry, then $\partial X$ is an anticanonical divisor in $X$.  See \cite[Section 5]{KLS2} for related discussion.

It is also convenient to introduce a slight weakening of the simplex-like condition.  We say that two positive geometries $(X,X_{\geq 0})$ and $(Y,Y_{\geq 0})$ are \emph{birationally isomorphic} if there is a birational isomorphism $f: X \dashrightarrow Y$, inducing a diffeomorphism $X_{\geq 0} \cong Y_{\geq 0}$ which respects the face stratification.  For example, $X$ could be a blow up of $Y$ along a subvariety that does not intersect $X_{\geq 0}$.  

We say that $(X, X_{\geq 0})$ is \emph{birationally simplex-like} if it is birationally isomorphic to a simplex-like positive geometry.

\subsection{Dimension 2}
Let us first consider normal positive geometries $(X,X_{\geq 0})$ with $X = \P^2$.  The normality condition implies that each boundary component is a smooth rational curve, and thus either a line or a conic.  Indeed, let $R \subset \P^2(\R)$ be a region that that is bounded by a simple closed curve $\gamma$, which is a union of finitely many pieces each of which is either a line segment or a part of a conic.  Most such regions $R$ are positive geometries.  For example, $R$ could be a convex polygon, or a non-convex polygon, or half of a disk; see \cref{fig:dimtwo}.  
%A non-example is to take $R$ to be a closed circular disk.  Then the boundary of $R$ is a circle, i.e., the whole of $\P^1(\R)$, which is not a positive geoemtry.

On the other hand, let $R$ be the closed unit disk.  Then $R$ is \emph{not} a positive geometry, since the boundary is a circle, the whole of $\P^1(\R)$, which is not a one-dimensional positive geometry.  

\begin{figure}
  \includegraphics[width=1.0\textwidth]{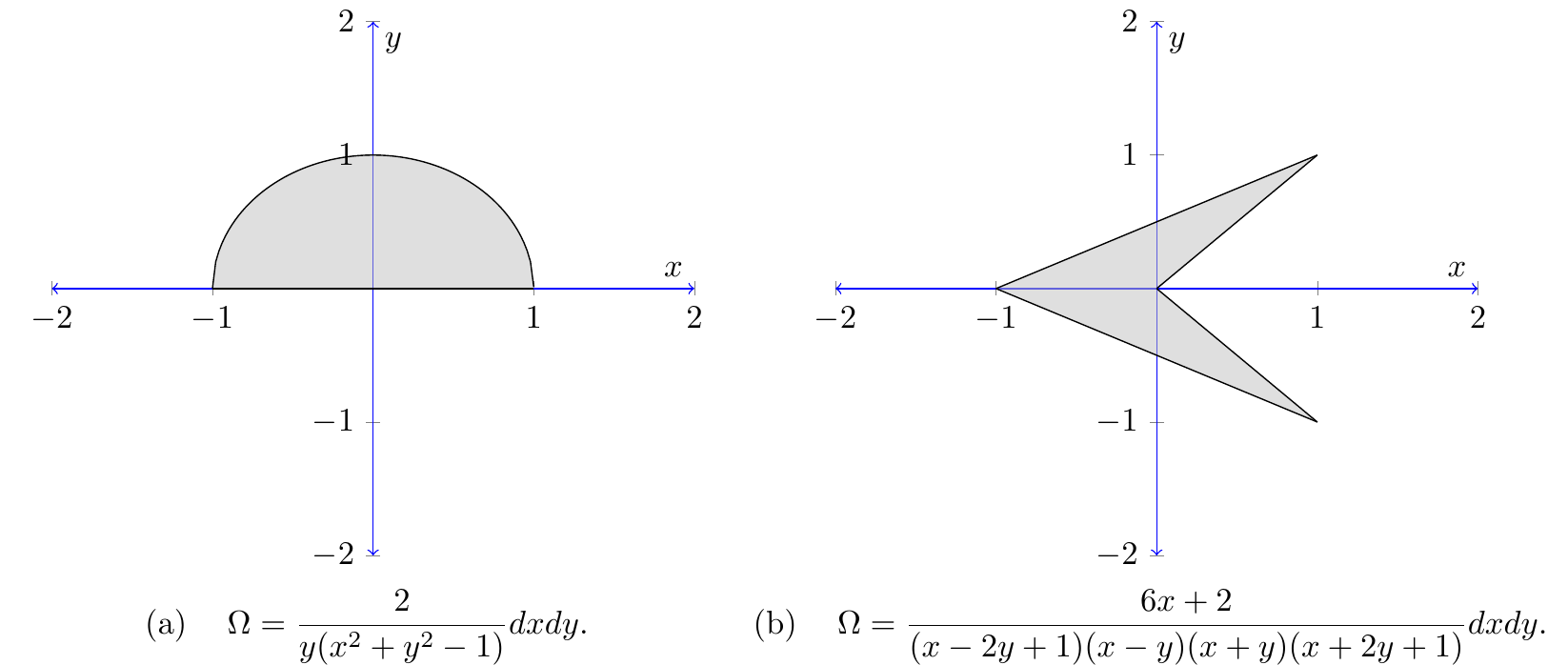}
 \caption{\label{fig:dimtwo} Some positive geometries in the plane that are not convex polygons.  (a): upper half unit disk.  (b): nonconvex 4-gon.}
\end{figure}

The list grows if we allow $(X,X_{\geq 0})$ to be non-normal.  A large zoo of examples arise from the theory of \emph{rational polypols} in the work of Kohn, Piene, Ranestad, Rydell, Shapiro, Sinn, Sorea, and Telen \cite{KPRRSSST}.  See \cref{fig:polypols} for some examples.  

Unfortunately, a classification of positive geometries in arbitrary dimensions seems out of reach.

\begin{figure}
  \includegraphics[width=1.0\textwidth]{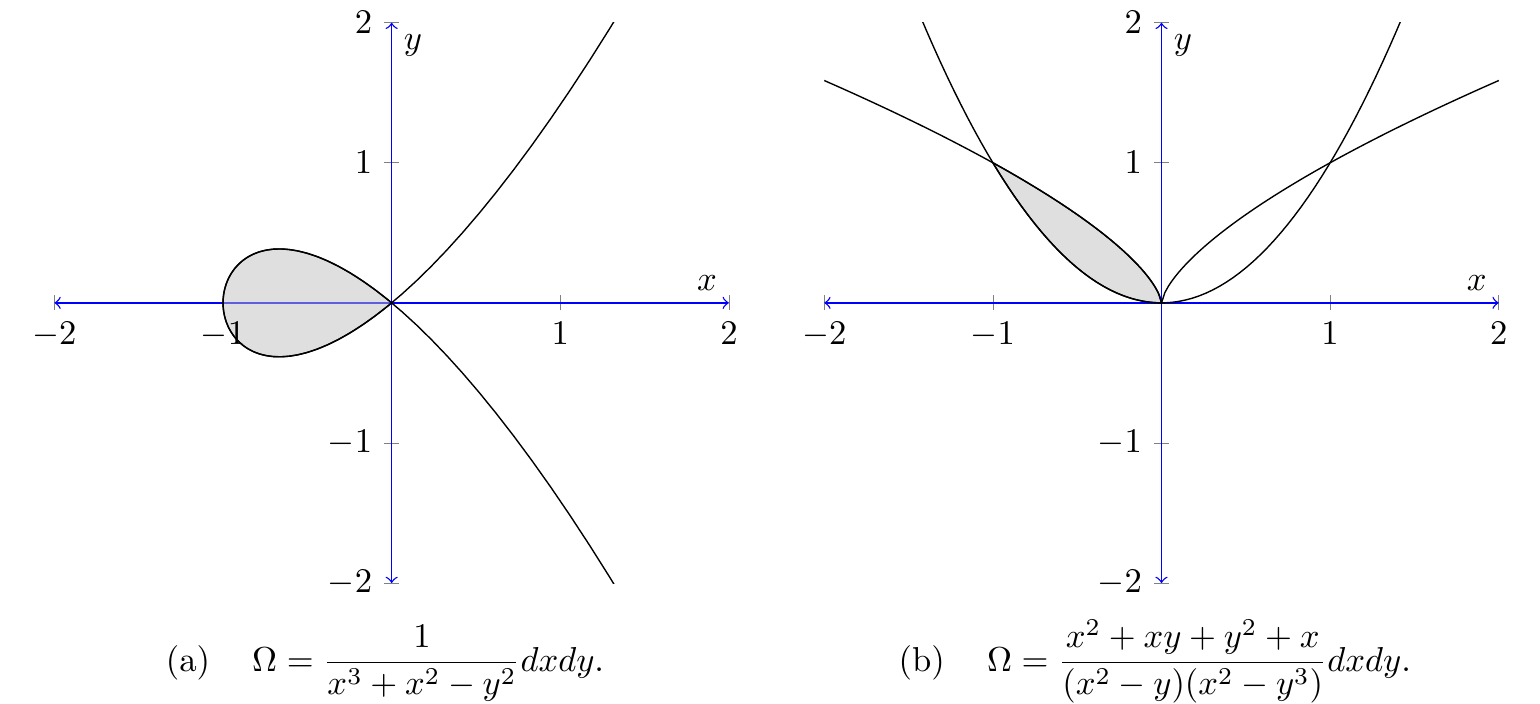}
 \caption{\label{fig:polypols} Some rational polypols.  (a): boundary nodal cubic with equation $y^2=x^2(x+1)$.  (b): boundary curves $y=x^2$ and $y^3=x^2$.}
\end{figure}

\subsection{Toric varieties}\label{sec:toric}
Let $X_P$ be the normal, projective toric variety associated to a $d$-dimensional lattice polytope $P$.  Then it has a natural positive part $X_{P,>0}$, defined to be the positive part $\R_{>0}^d$ of the torus $T(\R) \subset T$ that sits inside $X_P$.  We define the nonnegative part $X_{P,\geq 0}:= \overline{X_{P,>0}}$.

\begin{theorem}[\cite{ABL}]
$(X_P,X_{P,\geq 0})$ is a normal, simplex-like positive geometry.
\end{theorem}

The canonical form is the natural form $\Omega = \dlog x_1 \wedge \cdots \wedge \dlog x_n$ on $T$, extended to a rational top-form on $X_P$.  

While the combinatorics of faces for the two positive geometries, the polytope $P$, and the positive toric variety $X_{P,\geq 0}$ are identical, the algebraic geometry is quite different.  In particular, the canonical form $\Omega(X_{P, \geq 0})$ has no zeroes while $\Omega(P)$ typically has a very interesting zero set (\cref{thm:Padjoint}).  From this perspective, toric varieties are simpler positive geometries than polytopes.
%This is in some sense a simpler example of a positive geometry since $\Omega$ has no zeroes on $X_P$.

A \emph{morphism} $f: (X, X_{\geq 0}) \to (Y,Y_{\geq 0})$ of postive geometries is a rational map $f: X \dashrightarrow Y$ which restricts to a diffeomorphism $f|_{X_{>0}}: X_{>0} \to Y_{>0}$.  In \cite{ABL}, we formulated the heuristic that canonical forms should pushforward under morphisms of positive geometries.  \cref{thm:Ppush} is an example of this for a morphism from a toric variety to a polytope.  More precisely, the rational map $\Phi: \P^d \dashrightarrow \P^d$ of in \cref{thm:Ppush} can be factored into the rational maps
\begin{align*}
\alpha: &\P^d \dashrightarrow \P^{m-1} \qquad (1:z_1: \cdots:z_d) \longmapsto (\z^{\v_1}: \cdots : \z^{\v_m}) \\
\beta: &\P^{m-1} \dashrightarrow \P^d  \qquad (u_1:u_2: \cdots: u_m) \longmapsto \sum_{i=1}^m u_i \W_i.
\end{align*}
The (closure of the) image of $\alpha$ is the toric variety $X_\V$ associated to $\V$.  The linear map $\beta$ is a morphism of positive geometries from $X_{\V,\geq 0}$ to the polytope $P$.  In the case that $\V = \W$, the map $\beta$ is called the \emph{algebraic moment map}.

\subsection{Totally nonnegative Grassmannian and flag varieties}
Let $G$ be a split real reductive group and $P \subset G$ a parabolic subgroup.  Lusztig \cite{LusGP} has defined the totally nonnegative part $(G/P)_{\geq 0}$ of the partial flag variety $G/P$.  The following result was stated as an expectation in \cite{ABL}.

\begin{theorem}\label{thm:GP}
  $(G/P, (G/P)_{\geq 0})$ is a normal, simplex-like positive geometry.
\end{theorem}

For completeness, we give a proof of \cref{thm:GP} in \cref{sec:appendix}.

The face stratification of $(G/P)_{\geq 0}$ is well studied; see \cite{RieGP,LusGP}.  We have a decomposition $G/P = \bigsqcup_{[v,w]_P} \oPi_v^w$ where the \emph{open projected Richardson varieties} $\oPi_v^w$ \cite{KLS2} are indexed by equivalence classes of $P$-Bruhat intervals.  Denote by $\Pi_v^w = \overline{\oPi_v^w}$ the (closed) projected Richardson variety. The faces of $(G/P)_{\geq 0}$ are the positive geometries $(\Pi_v^w, (\Pi_v^w)_{\geq 0})$, where $(\Pi_v^w)_{\geq 0}:= \Pi_v^w \cap (G/P)_{\geq 0}$.
%with the structure of a regular CW-complex homeomorphic to a closed ball \cite{}.  

The case of the Grassmannian $\Gr(k,n)$ of $k$-planes in $\C^n$ is of particular importance in numerous situations.  Postnikov \cite{Pos} gave an independent definition of the totally nonnegative part $\Gr(k,n)_{\geq 0}$, as the subspace represented by $k \times n$ matrices all of whose $k \times k$ minors are nonnegative.  For $k = 1$, we have $\Gr(1,n)_{\geq 0} = \P^{n-1}_{\geq 0}$ is a $n-1$ dimensional simplex.

\begin{theorem}\label{thm:Gr}
$(\Gr(k,n), \Gr(k,n)_{\geq 0})$ is a normal, simplex-like positive geometry.
\end{theorem}
The faces of $(\Gr(k,n), \Gr(k,n)_{\geq 0})$ are positive geometries called \emph{positroid cells}, and their Zariski closures are called \emph{positroid varieties} \cite{Pos, KLS, CDM}.

% This is quite a non-trivial result, and requires combining results from \cite{Pos, KLS, Book, CDM}.  Part of the non-triviality of this result is that it requires a complete description of all the positive geometries that appear as facets, called {\it positroid cells} (\cite{Pos}), and properties such as normality of their Zariski-closures, called {\it positroid varieties} (\cite{KLS}).

\begin{remark}
Conjecturally, the space of planar Ising models \cite{GP}, also called the positive orthogonal Grassmannian \cite{HW}, and the space of electrical networks \cite{LamElec}, also called the positive Lagrangian Grassmannian \cite{CGS,B} are positive geometries.  %These are certain subspaces of the totally nonnegative Grassmannian associated to a (symmetric or skew-symmetric) bilinear form.
\end{remark}

We expect that other spaces appearing in total positivity, such as wonderful compactifications or Kac-Moody flag varieties \cite{BH,He}, are positive geometries.

\subsection{Moduli space}
Let $X = \bM_{0,n}$ be the Deligne-Knudsen-Mumford compactification \cite{DM} of the moduli space of $n$ points on $\P^1$.  This is a smooth complex projective variety of dimension $n - 3$.  The open subset $\bM_{0,n} \subset X$ is the moduli-space of $n$ distinct points on $\P^1$.  It is isomorphic to a hyperplane arrangement complement.   The real part $\bM_{0,n}(\R)$ is a $(n-3)$-dimensional smooth manifold with $(n-1)!/2$ connected components.  We define $(\bM_{0,n})_{\geq 0} \subset X$ to be the closure in $X$ of one of these connected components.  (The group $S_n$ acts on $X$ by permuting the $n$ points, and this action acts transitively on the connected components of $\bM_{0,n}(\R)$.)

\begin{theorem}[{\cite[Proposition 8.2]{AHLcluster}}]
$(\bM_{0,n},(\bM_{0,n})_{\geq 0} )$ is a positive geometry.
\end{theorem}

For a discussion of differential forms on $\bM_{0,n}$, see for example \cite{BCS}.

The positive geometry $(\bM_{0,n})_{\geq 0}$ is diffeomorphic to the associahedron polytope as a stratified space.  This example is not a simplex-like geometry.  However, it is birationally simplex-like.  Indeed, it is birationally isomorphic to the toric variety $(X_P, X_{P,\geq 0})$ associated to the associahedron; see \cite[Section 10]{AHL}.  For closely related geometries, see the cluster configuration spaces of \cite{AHLcluster} and the positive Chow cells of \cite{ALS}.

\subsection{Cluster varieties}
Let $Y$ be a $d$-dimensional cluster variety, that is, the spectrum $Y = \Spec(A)$ of a cluster algebra of geometric type.  For cluster algebras, we follow the convention that frozen variables are inverted \cite{LS}.

We will assume that $Y$ is a reasonable cluster variety, for example, locally acyclic and of full rank \cite{Mul}.  Then $Y$ has a natural positive part $Y_{>0}$, the positive part of any cluster torus in $Y$.  Furthermore, $Y$ has a natural top form $\Omega(Y_{>0}):= \prod_i dx_i/x_i$, defined up to sign, where $(x_1,\ldots,x_d)$ is a seed of the cluster algebra.

\begin{conjecture}
  Let $Y$ be a cluster variety, locally acyclic and of full rank.  Then there is a compactification $X$ of $Y$ such that $(X,X_{\geq 0})$ is a simplex-like positive geometry, where:
  \begin{enumerate}
    \item $X_{>0} = Y_{>0}$ and $Y = X \setminus \partial X$,
    \item $\Omega(X_{\geq 0}) = \Omega(Y_{>0})$,
    \item each face positive geometry $(C,C_{\geq 0})$ of $(X,X_{\geq 0})$ is also a compactification of a cluster variety (and its positive part).
  \end{enumerate}
\end{conjecture}

\cref{thm:Gr} is the prototypical example of this conjecture.  Indeed, the stratification of the Grassmannian induced by positroid cells are the open positroid varieties, which are shown to be cluster varieties in \cite{GLcluster}.

\subsection{Grassmann polytopes and amplituhedra}\label{sec:GrassP}
Let $Z: \R^n \to \R^{k+m}$ be a linear map and assume that $n \geq k+m$.  Then generically, for a $k$-dimensional subspace $V \subset \R^n$, we have that $Z(V) \subset \R^{k+m}$ is again $k$-dimensional.  This induces a rational map $Z: \Gr(k,n) \dashrightarrow \Gr(k,k+m)$.  Assuming that $Z$ is well-defined on $\Gr(k,n)_{\geq 0}$, the image $P = Z(\Gr(k,n)_{\geq 0})$ is a \emph{Grassmann polytope} \cite{CDM, Karp}.  Viewing $Z$ as a $n \times (k+m)$ matrix, we call $Z$ positive when all the $(k+m)\times (k+m)$ minors of $Z$ are positive.  In this case, $A_{n,k,m} = Z(\Gr(k,n)_{\geq 0})$ is known as the {\it amplituhedron}.

In the case $k = 1$, Grassmann polytopes are projective polytopes, and the amplituhedron is a cyclic polytope.  Thus the following conjecture holds for $k=1$.

\begin{conjecture}[\cite{ABL}] \label{conj:Grass}
  Grassmann polytopes and amplituhedra are positive geometries.
\end{conjecture}  

The (conjectural) canonical form of the amplituhedron $A_{n,k,m}$ is very well-studied due to its relation to super Yang-Mills amplitudes \cite{AT}.  Among the remarkable properties the amplituhedron canonical form satisfies, let us mention the conjectural positivity \cite{AHTpos} and the behavior under parity duality \cite{GLparity}. 

Grassmann polytopes give a wealth of examples of positive geometries that are not simplex-like.  Some other (potential) examples are naive positive parts of flag varieties \cite{BK,BHL}, higher-loop amplituhedra \cite{ABL,AT2}, and momentum amplituhedra \cite{DFLP}.  These examples, as well as \cref{conj:Grass}, may require extending the notion of positive geometry to allow for boundary components to be formal sums of positive geometries; see the recent paper \cite{DHS}.  

\section{Combinatorics, topology, etc.}
We pose a number of questions on various features positive geometries. 

\subsection{Positive topology}
A positive geometry $X_{\geq 0}$ is \emph{connected} if $X_{>0}$ is a connected manifold and all boundary components $(C,C_{\geq 0})$ are connected positive geometries.  Note that this is stronger than the condition that $X_{>0}$ is connected as a topological space.  For simplicity, we restrict our discussion to connected positive geometries.

\begin{problem}  What can we say about the topology of $X_{\geq 0}$ and $X_{>0}$ for connected positive geometries?
\end{problem}

Polytopes, positive parts of toric varieties, totally positive parts of flag varieties, and the positive part of the moduli space of $n$ points are all examples of connected positive geometries.    In all these cases, 
\begin{enumerate}
  \item $X_{>0}$ is homeomorphic to an open ball, and $X_{\geq 0}$ is homeomorphic to a closed ball of the same dimension;
  \item the face poset of $X_{\geq 0}$ is a Eulerian and shellable;
  \item the face stratification of $X_{\geq 0}$ endows it with the structure of a regular CW-complex.
\end{enumerate}

For polytopes, these statements are well-known.  The toric variety case and the moduli space case are both diffeomorphic to polytopes.  For the totally positive flag varieties, see \cite{GKL,GKL2} for (1), \cite{Wil} for (2), and \cite{GKL3} for (3).  It would be interesting to establish (1),(2),(3) for connected Grassmann polytopes.  See also \cite{BGPZ}.

%It is tempting to compare the existence and uniqueness of the canonical form with the fact that the relative homology $H_n(B^n,S^{n-1})$ is one-dimensional.

\subsection{Complex topology}
\def\oX{{\mathring{X}}}
Let $(X,X_{\geq 0})$ be a positive geometry, and let $\oX:= X \setminus \partial X$ be the complex algebraic ``open stratum'' of the positive geometry.  For simplicity, we assume that $\oX$ is a smooth complex algebraic variety.  

Since $\Omega(P)$ has poles only along $\partial X$, it is holomorphic on $\oX$.  Furthermore, $\Omega(P)$ is a holomorphic top-form and thus closed.  It therefore defines a class $[\Omega(P)] \in H^d(\oX) = H^d(\oX,\C)$ in the deRham (and therefore singular) cohomology of $\oX$.

\begin{problem}
  What can we say about the topology of $\oX$ and what can we say about the class $[\Omega(P)] \in H^d(\oX)$?
\end{problem}

\begin{enumerate}
  \item[(a)] 
For $X_{\geq 0} = P$ a polytope, $\oX$ is a hyperplane arrangement complement.  
\item[(b)] 
For $X_{\geq 0} = X_{P,\geq 0}$ the positive part of a toric variety, $\oX$ is a complex algebraic torus.  
\item[(c)] For $X_{\geq 0} = (G/P)_{\geq 0}$ a totally nonnegative flag variety, $\oX$ is the top open projected Richardson stratum. 
\end{enumerate}
The cohomology of hyperplane arrangement complements are well-understood from both geometric and combinatorial perspectives \cite{OT}.
The cohomologies of open Richardson varieties were studied in \cite{GLCatalan}.  In the case of an open positroid variety, or more generally a Grassmann polytope or amplituhedron, the space $\oX$ could be considered a Grassmannian variant of a hyperplane arrangement complement.  These cohomologies are especially interesting in the positroid case, where they are related to $q,t$-Catalan numbers.

Curiously, in the simplex-like positive geometry discussed in \cref{sec:examples}, we have $\dim H^d(\oX) = 1$, and the cohomology group is spanned by $[\Omega(P)]$.  

\begin{enumerate}
  \item[(a)] 
For $X_{\geq 0} = P$ a simplex, we have $\oX = (\C^*)^d$ which deformation retracts to a torus $(S^1)^d$, whose cohomology is well-known.  For \cref{fig:dimtwo}(a), the formula $H^d(\oX) = \C \cdot [\Omega(P)]$ can be proven, for example, using the Gysin long exact sequence.
\item[(b)] 
For the toric variety case, we again have $\oX = (\C^*)^d \simeq (S^1)^d$.
\item[(c)] For $\oX$ an open projected Richardson variety, the dimension $\dim(H^d(\oX))=1$ is given in \cite[Proposition 8.6]{HLZ}.  Let us show that $[\Omega(X_{\geq 0})] \neq 0$ inside $H^d(\oX)$.  In the notation of \cref{sec:appendix}, for $[v',w']_P \lessdot [v,w]_P$, we have a Gysin exact sequence 
$$
\cdots \to H^d(\oPi_{[v,w]_P} \cup \oPi_{[v',w']_P}) \to H^d(\oPi_{[v,w]_P}) \stackrel{\Res}{\longrightarrow} H^{d-1}(\oPi_{[v',w']_P}) \to \cdots.
$$
By induction on dimension we may assume that $[\Omega((\Pi_{[v',w']_P})_{\geq 0})]$ spans $H^{d-1}(\oPi_{[v',w']_P})$; then \eqref{eq:Res} in \cref{sec:appendix} shows that $[\Omega((\Pi_{[v,w]_P})_{\geq 0})]$ spans $H^{d-1}(\oPi_{[v,w]_P})$.
\end{enumerate}

The fact that $H^d(\oX)$ is spanned by $[\Omega(P)]$ also suggests a relation to mirror symmetry \cite{LT,HLZ}.  See also \cite{LS} for a discussion of the cluster variety case.

On the other hand, for a general polytope, or a positive geometry that is not simplex-like, $H^d(\oX)$ can be large.

\subsection{Triangulations}
In \cref{sec:poly}, we give many formulae for the canonical form of a polytope.  A fundamental problem in positive geometries is to give (similar or otherwise) formulae for canonical forms of positive geometries.

A collection $(X_1,X_{1,\geq 0}), (X_2,X_{2,\geq 0}), \ldots,(X_r,X_{r,\geq 0})$ of positive geometries is a \emph{subdivision} of a positive geometry $(X,X_{\geq 0})$ if the interiors $X_{i,>0}$ are pairwise disjoint and $\bigcup_i X_{i,\geq 0} = X_{\geq 0}$.  (This is weaker than the usual notion of subdivion for polytopes.)  \cref{eq:triang} generalizes to positive geometries.

\begin{thm}[\cite{ABL}] \label{thm:triang}
  If $(X_1,X_{1,\geq 0}), (X_2,X_{2,\geq 0}), \ldots,(X_r,X_{r,\geq 0})$ subdivides $(X,X_{\geq 0})$ then $\Omega(X_{\geq 0})= \sum_i \Omega(X_{i,\geq 0})$.
\end{thm}

In the case of the amplituhedron, \cref{thm:triang} is the main tool used to construct the canonical form \cite{AT,ATT}.  This has led to much work on the triangulations of amplituhedra; see the recent works \cite{GLparity,PSW,ELT} and references therein.

\begin{conjecture}\label{conj:triang}
  Every positive geometry has a subdivision into positive geometries that are birationally simplex-like.
\end{conjecture}
A subdivision as in \cref{conj:triang} is called a triangulation.  We believe the adjective ``birationally" is necessary because of the possibility of modifying $X$ by a blowup.    

\subsection{Adjoint}
Motivated by \cref{thm:Padjoint}, we make the following definition.  See \cite{KPRRSSST} for a detailed discussion.

\begin{definition}
The \emph{adjoint hypersurface} $A_{X_{\geq 0}}$ of a positive geometry $(X,X_{\geq 0})$ is the closure of the zero locus of $\Omega(X_{\geq 0})$.  
\end{definition}

Up to a scalar, determining the adjoint hypersurface is equivalent to determining the canonical form of a positive geometry.  Let $C_1,C_2,\ldots,C_r$ be the irreducible components of $\partial X$.  We produce a finite list of closed irreducible subvarieties of $X$ by repeatedly intersecting and taking irreducible components, starting with $C_1,\ldots,C_r$.  Define the \emph{residual arrangement} $\RR_{X_{\geq 0}}$ to be the collection of such subvarieties that do not intersect $X_{\geq 0}$.  

% \begin{conjecture}
%   Let $(X,X_{\geq 0})$ be a positive geometry.  Then the adjoint hypersurface $A_{X_{\geq 0}}$ contains the residual arrangement $\RR_{X_{\geq 0}}$.
% \end{conjecture}

\begin{proposition}\label{prop:adj}
  Let $(X,X_{\geq 0})$ be a positive geometry and assume that all the varieties in the face stratification of $X_{\geq 0}$ are smooth.  Then the adjoint hypersurface $A_{X_{\geq 0}}$ of $X_{\geq 0}$ contains $\RR_{X_{\geq 0}}$.
\end{proposition}
\begin{proof}
  Let $C$ be a boundary component of $(X,X_{\geq 0})$.
  We first show that $A_{C_{\geq 0}}$ is contained in $A_{X_{\geq 0}} \cap C$.  Using the smoothness assumption, we let $f_1,f_2,\ldots,f_d$ be local coordinates for $X$ such that $C$ is cut out by $f_1 = 0$ to order one.  Locally, $df_1 \wedge df_2 \wedge \cdots \wedge df_d$ is a non-vanishing top form, so we have
  $$
  \Omega(X_{\geq 0}) = \frac{A_{X_{\geq 0}}}{f_1 p_2 p_3 \cdots p_m} df_1 \wedge df_2 \wedge \cdots \wedge df_d,
  $$
  where $p_2,p_3,\ldots$ are regular functions whose vanishing sets are the other boundary components of $X$, and $A_{X_{\geq 0}}$ here denotes a regular function that vanishes along the adjoint hypersurface.  Then 
  \begin{equation}\label{eq:ResC}
  \Omega_{C_{\geq 0}}= \Res_{C}  \Omega(X_{\geq 0}) = \left(\frac{A_{X_{\geq 0}}}{p_2 p_3 \cdots p_m}\right) \bigg|_C df_2 \wedge \cdots \wedge df_d.
  \end{equation}
  The restriction to $C$ is obtained by setting $f_1= 0$.  After cancelling out common factors in the numerator and denominator, we deduce that the adjoint hypersurface $A_{C_{\geq 0}}$ is contained in $A_{X_{\geq 0}} \cap C$.
 
  Now let $Z \subset \RR_{X_{\geq 0}}$ be an irreducible component and let $C$ be a boundary component of $X$ that contains $Z$.  Let $D_1,D_2,\ldots,D_r$ be the irreducible components of $C_i \cap C_j$ for two boundary components of $(X,X_{\geq 0})$.  Since $X$ is smooth, all these subvarieties are codimension two \cite[\href{https://stacks.math.columbia.edu/tag/0AZL}{Tag 0AZL}]{stacksproject}.  Reindex so that $D_1,D_2,\ldots,D_p$ are the boundary components of $(C,C_{\geq 0})$.  Suppose first that $Z$ belongs to one of the subvarieties $D_{p+1}, \ldots, D_r$. Since $Z \subset C$, we may assume that $Z \subset D \subset C$ where $D$ is an irreducible component of $C \cap C'$ for $C'$ some other boundary component of $(X,X_{\geq 0})$.  But $D$ is not an irreducible component of $\partial C$, so $\Omega(C_{\geq 0})$ does not have a pole along $D$.  On the other hand, $\Omega(X_{\geq 0})$ has a pole along $C'$.  From \eqref{eq:ResC}, we see that this is only possible if $A_{X_{\geq 0}}$ vanishes along $D$.
  
 Next, suppose that $Z$ does not belong to any of $D_{p+1}, \ldots, D_r$.  Then $Z$ can be obtained by repeatedly intersecting and taking irreducible components of the subvarieties $D_1,D_2,\ldots,D_p$.  On the other hand, $Z \cap X_{\geq 0} = \emptyset$ which implies that $Z \cap C_{\geq 0} = \emptyset$, so by definition we have $Z \subset \RR_{C_{\geq 0}}$.  By induction on dimension, we conclude that $Z \subset A_{C_{\geq 0}}$, and since $A_{C_{\geq 0}} \subset A_{X_{\geq 0}} \cap C \subset A_{X_{\geq 0}}$, the result follows.
\end{proof}

We expect \cref{prop:adj} to hold without any smoothness assumption.

\begin{problem}\label{prob:res}
  When is $A_{X_{\geq 0}}$ characterized by the property of containing $\RR_{X_{\geq 0}}$?
\end{problem}

In the case that $X_{\geq 0}$ is one-dimensional, \cref{prob:res} has an affirmative answer only when $X_{\geq 0}$ is a closed interval.  When $X_{\geq 0}$ is a disjoint union of multiple intervals, the canonical form has a non-trivial zero set, but the residual arrangement is empty.

Note that even in the case of polytopes, the adjoint hypersurface is not always uniquely determined by its vanishing on the residual arrangement \cite{KR}.  

It would be especially interesting to study the residual arrangement of the amplituhedron.  %See \cite{AHTpos}.

\subsection{Positive convexity and dual positive geometries}
%For a polytope, the canonical form can be expressed as the volume of the dual polytope.   We are thus led to to the following subclass of positive geometries.
A positive geometry $(X,X_{\geq 0})$ is called {\it positively convex} \cite{ABL} if $\Omega(X,X_{\geq 0})$ has constant sign on $X_{>0}$.  In other words, the poles and zeros of $\Omega(X,X_{\geq 0})$ do not intersect $X_{>0}$.  For a polytope $P$, the dual volume formula (\cref{thm:Pdual}) shows that the canonical form $\Omega(P)$ takes constant sign in the interior of the polytope. Thus polytopes are positively convex positive geometries.  \cref{fig:dimtwo}(a) and \cref{fig:polypols}(a,b) are positively convex geometries.  \cref{fig:dimtwo}(b) is not positively convex.  It is conjectured in \cite{ABL,AHTpos} that certain amplituhedra are positively convex geometries.

As suggested by \cref{thm:Pdual}, positive convexity indicates the potential existence of a \emph{dual positive geometry}.

%This is a substitute for the usual notion of convexity for subspaces of $\R^n$.

\subsection{Integral functions}
Let us say only a brief word about the relation between positive geometries and scattering amplitudes which typically arises from taking integrals of the canonical form over positive geometries.
One typically considers integrals of the form
\begin{equation} \label{eq:integral}
\int_{X_{>0}} (\text{regulator}) \Omega(X_{\geq 0}) 
\end{equation}
Since $\Omega(X_{\geq 0})$ has poles along $\partial X_{\geq 0}$, some regulator in the integrand is necessary for the integral to converge.  Different regulators appear in different applications.

The motivating examples of \eqref{eq:integral} are integrals for scattering amplitudes in physics.  We refer the reader to \cite{HT,FL} for further discussion.

Let us mention two classes of examples that recover classical special functions.  The \emph{stringy canonical forms} of \cite{AHL} are obtained by taking the regulator to be a monomial $f_1^{s_1} \cdots f_r^{s_r}$ in a collection of rational functions $f_1,f_2,\ldots,f_r$, and viewing the integral as a function of the exponents $s_1,s_2,\ldots,s_r$.  Stringy canonical forms are so named because in the case $X_{\geq 0} = (\bM_{0,n})_{\geq 0}$ one obtains string theory amplitudes.  For $n=4$, one recovers the beta function, studied clasically by Euler and Legendre, as a special case.  Curiously, these integrals also appear as marginal likelihood integrals in algebraic statistics \cite{ST}.

Another class of examples appear in mirror symmetry and the theory of geometric crystals \cite{Rie,LamGB,LT}.  In this case, one takes the regulator to be $\exp(f)$ for a rational function $f$ called the \emph{superpotential}.  When $(X,X_{\geq 0}) = (\P^1,\Delta^1)$, this integral recovers the Bessel function, studied classically by Bernoulli, Bessel, and others.

\appendix
\section{Proof of \cref{thm:GP}}\label{sec:appendix}

%Let $G/P$ denote a generalized partial flag variety and $(G/P)_{\geq 0}$ its totally nonnegative part.  
Write $[v',w']_P \leq [v,w]_P$ for the partial order on projected Richardson varieties, and let $[v',w']_P \lessdot [v,w]_P$ denote a cover relation.  Thus $(\Pi^{w'}_{v'})_{\geq 0}  \subseteq (\Pi^w_v)_{\geq 0}$ for $[v',w']_P \leq [v,w]_P$.  We generally follow the notations of \cite{KLS2,GKL3}.

The projected Richardson variety $\Pi_v^w$ is irreducible, normal, and of dimension $\ell(w)-\ell(v)$ \cite{KLS2}.  Write $\partial \Pi_v^w$ for the union of all projected Richardson varieties contained in $\Pi_v^w$ but of smaller dimension.  The irreducible components of $\partial \Pi_v^w$ are called the facets of $\Pi_v^w$.
In \cite[Section 5]{KLS2}, it is shown that $\partial \Pi_v^w$ is an anticanonical divisor in $\Pi_v^w$.  It follows that there is a rational top-form $\omega_v^w$, unique up to scalar, that is holomorphic and nonzero on $\oPi_v^w$, with simple poles along the facets of $\Pi_v^w$.  Furthermore, for $[v',w']\lessdot[v,w]$, we have
\begin{equation}\label{eq:Res1}
\Res_{\Pi_{v'}^{w'}} \omega_v^w = \omega_{v'}^{w'} \text{ up to scalar.}
\end{equation} 

To prove \cref{thm:GP}, we define a top-form $\Omega_v^w$ that is a scalar multiple of $\omega_v^w$ and that satisfies the defining recursion for the canonical form of $(\Pi_v^w, (\Pi_v^w)_{\geq 0})$.  Each $\oPi_v^w$ is isomorphic to an open Richardson variety in the full flag variety $G/B$, and it is enough to consider the $G/B$ case.  Then the strata are indexed by intervals $[v,w]$, $v \leq w$ in Bruhat order.

By \cite{GKL3}, $(G/B)_{\geq 0}$ is a stratified closed ball.  We may pick an orientation on the closed ball $(G/B)_{\geq 0}$ that is compatible with orientations on each of the strata.  Thus it is enough to define $\Omega_v^w$ up to sign: the sign can be fixed by requiring that $\Omega_v^w$ induces the chosen orientation.

Let $v \leq w$ and $\w$ be a reduced word for $w$.  Let $\v$ be the positive distinguished subexpression of $\w$, that is, $\v$ is the rightmost subword of $\w$ that is a reduced word for $v$.  Let $J \subset [1,\ell(w)]$ denote the indices not belonging to the subword $\v$.  For $\t \in (\C^*)^{|J|}$, we have a group element (see \cite{GKL3} for notation)
\begin{equation}\label{eq:g}
g_{\v,\w}(\t) = g_1 g_2 \cdots g_n \in G
\end{equation}
where $g_k = \dot s_{i_k}$ or $g_k = y_{i_k}(t_k)$ depending on whether $k \notin J$ or $k \in J$.  The map $\t \mapsto g_{\v,\w}(\t)B \in G/B$ defines a rational parametrization of $\Pi_v^w$.  Define the rational form
$$
\Omega_v^w = \pm \bigwedge_{k \in J} \frac{dt_k}{t_k}
$$
on $\Pi_v^w$.  It is not immediately clear that $\Omega_v^w$ is holomorphic on $\oPi_v^w$.  By \cite[Proof of Proposition 7.2]{Rie} and \cite[Proof of Proposition 5.4]{GKL3}, up to sign the form $\Omega_v^w$ does not depend on the choice of $\w$.

We shall show that 
\begin{align}
&\mbox{$\Omega_v^w$ is equal to $\omega_v^w$ up to a scalar.} \label{eq:uscalar}\tag{$\star$}\\
& \Res_{\Pi_{v'}^{w'}} \Omega_v^w = \pm \Omega_{v'}^{w'} \text{ for }[v',w'] \lessdot [v,w] \label{eq:Res}\tag{$\star\star$}
\end{align}
which together proves \cref{thm:GP}.

By \cite[Proposition 2.11]{LamGB}, \eqref{eq:uscalar}
 holds for $(v,w) = (e,w_0)$.  Now let $(v,w)= (e,w)$.  Then we can choose $\wo = \w_1 \w_2$ where $\w_2$ is a reduced word for $w$.  Then $g_{e,\wo}(\t) = g_1(t_1,\ldots,t_r) g_2(t_{r+1},\ldots,t_{\ell})$ for $\ell= \ell(w_0)$ and $r = \ell(w_0)-\ell(w)$, and $g_2 = g_{e,\w}$ parametrizes $\Pi_e^w$.  Thus $\Omega_1^w$ is obtained from $\Omega_1^{w_0}$ by repeatedly taking residues $\Res_{t_k=0} \circ \cdots \circ \Res_{t_2=0} \circ \Res_{t_1=0}$.  By \eqref{eq:Res1}, this proves \eqref{eq:uscalar} for the $v=e$ case.  Similarly, when $(v,w) = (v,w_0)$, we pick $\wo = \w_1 \w_2$ where $\w_2$ is a reduced word for $v$.  A calculation shows that in the parametrization given by the group element \eqref{eq:g}, $\lim_{t_k \to \infty}$ sends $y_{i_k}(t_k)$ to $\dot s_{i_k}$.  Then $\Omega_v^{w_0}$ is obtained from $\Omega_e^{w_0}$ by repeatedly taking residues $\cdots \circ \Res_{t_{\ell-1}=\infty} \circ \Res_{t_\ell=\infty}$.  By \eqref{eq:Res1}, this proves \eqref{eq:uscalar} for the $w=w_0$ case.

 Next, suppose that we have $[v',w'] \lessdot [v,w]$ where $v,w$ are arbitrary.  Then either (a) $w'=w$ or (b) $v'=v$.  Suppose we are in case (a).  If $w'=w = w_0$, we may pick $\wo= \w_1\w_2$ where $\w_2$ is a reduced word for $v'$.  Then \eqref{eq:Res} follows as before.  Otherwise if $w \neq w_0$, we let $w_0 = u w$ and left multiply the parametrization $g_{\v,\w}(\t)$ by $y_{i_1}(a_1)\cdots y_{i_r}(a_r)$ where $u= s_{i_1} \cdots s_{i_r}$ is reduced, obtaining a parametrization $g_{\v,\wo}$, and similarly for $g_{\v',\wo}$.  Using the definition of residue, we have
 $$
\bigwedge_i \frac{da_i}{a_i} \wedge \Omega_{v'}^w = \pm \Omega_{v'}^{w_0}= \pm \Res_{\Pi^{w_0}_{v'}} \Omega_v^{w_0} =  \pm \Res_{\Pi^{w_0}_{v'}} \left(\bigwedge_i \frac{da_i}{a_i} \wedge \Omega_v^w\right) =  \pm  \bigwedge_i \frac{da_i}{a_i} \wedge \Res_{\Pi^{w}_{v'}} \Omega_v^w.
$$ 
Wedging with $\bigwedge_i \frac{da_i}{a_i}$ is injective for forms on $\oPi_{v'}^w$, so we deduce that \eqref{eq:Res} holds in case (a).  By \eqref{eq:Res1}, this implies that \eqref{eq:uscalar} holds for $[v',w]$ if it holds for $[v,w]$.  Since we have shown \eqref{eq:uscalar} for $(v,w) = (e,w)$, we deduce that \eqref{eq:uscalar} holds for all $[v,w]$.
   
There is an automorphism $\phi:G/B \to G/B$ that takes $\oPi_v^w$ to $\oPi^{vw_0}_{ww_0}$ for each $[v,w]$; see \cite{LusztigSpringer}.  Since $\Omega_v^w$ satisfies \eqref{eq:uscalar}, we have that $\phi^*(\Omega_v^w)$ is a scalar multiple of $\Omega_{ww_0}^{vw_0}$.  This scalar multiple must be $\pm 1$, since both $\phi^*(\Omega_v^w)$ and $\Omega_{ww_0}^{vw_0}$ have unit residues on some $0$-dimensional strataum $\Pi_u^u$.  Thus the validity of \eqref{eq:Res} for case (b) follows from that for case (a).  We are done.

\end{document}